\documentclass[11pt,a4paper]{paper}
\usepackage{amsfonts,amsmath,amsthm,amssymb,indentfirst}

\usepackage{mathtools}

\usepackage[utf8]{inputenc}
\usepackage{graphicx}

\usepackage{color}

\newcommand{\ce}{\mathbb{C}}

\newcommand{\er}{\mathbb{R}}
\newcommand{\R}{\mathbb{R}}
\newcommand{\N}{\mathbb{N}}

\usepackage{hyperref}

\newtheorem{Teorema}{Theorem}[section]

\newtheorem{Rem}[Teorema]{Remark}
\newtheorem{Prop}[Teorema]{Proposition}
\newtheorem{Lemma}[Teorema]{Lemma}

\DeclareMathOperator{\sech}{sech}

\title{On a coupled system of a Ginzburg-Landau equation with a quasilinear conservation law}
\author{Jo\~ao-Paulo Dias\thanks{\textit{jpdias@fc.ul.pt},  Departamento de Matem\'atica e CMAFCIO, Faculdade de Ci\^encias da Universidade de Lisboa,  Campo Grande, Edif\'icio C6, 1749-016 Lisboa (Portugal)}, Filipe Oliveira\thanks{\textit{foliveira@iseg.ulisboa.pt}, Departamento de Matem\'atica e CEMAPRE, ISEG, Universidade de Lisboa,
		Rua do Quelhas 6, 1200-781 Lisboa (Portugal)}, and Hugo Tavares\thanks{ \textit{hrtavares@ciencias.ulisboa.pt} Departamento de Matem\'atica e CMAFCIO, Faculdade de Ci\^encias da Universidade de Lisboa, Campo Grande, Edif\'icio C6, 1749-016 Lisboa (Portugal)} }

\begin{document}
\maketitle

\begin{abstract}
	\noindent
We study a coupled system of a complex Ginzburg--Landau equation with a quasilinear conservation law
\begin{equation*}
\left\{\begin{array}{rlll}
e^{-i\theta}u_t&=&u_{xx}-|u|^2u-\alpha g(v)u&\\ \\
v_t+(f(v))_x&=&\alpha (g'(v)|u|^2)_x&
\end{array}\right.
,\qquad x\in\er,\, t	\geq 0
\end{equation*}
which can describe the interaction between a laser beam and a fluid flow (see [Aranson, Kramer, Rev. Med. Phys. 74 (2002)]). We prove the existence  of a local in time strong solution for the associated Cauchy problem and, for a certain class of flux functions, the existence of global weak solutions. Furthermore we prove the existence of standing wave solutions of the form $(u(t,x),v(t,x))=(U(x),V(x))$ in several cases.
\end{abstract}
\noindent \textbf{Keywords:} Ginzburg-Landau, Conservation Laws, Standing Waves, Elliptic Problems, Variational Methods, Nehari Manifold.

 \bigskip

\noindent \textbf{Mathematics Subject Classification (2010):} 		35Q56, 35L65, 	35J20

\begin{section}{Introduction}

In  \cite{DF},  \cite{DFF} and \cite{DFO} we studied the well-posedness of several universal models describing the interaction between long and short waves. These unidimensional systems consist on the coupling between a nonlinear Sch\"odinger equation and a conservation law, and can be put in the general form
\begin{equation}
\left\{\begin{array}{rlll}
iu_t&=&u_{xx}-|u|^2u-\alpha g(v)u&\\ \\
v_t+(f(v))_x&=&\alpha (g'(v)|u|^2)_x&
\end{array}\right.
,\qquad x\in\er,\, t	\geq 0,
\end{equation}
where $f$ and $g$ are regular functions, $u$ is complex-valued (the transverse component of a field in complex notation) and $v$ is real-valued (a concentration). These models, originally derived by Benney (\cite{B}) in the case where $g$ is linear, have been successfully applied to several physical contexts. In water waves theory, applications include the interaction between gravity-capillary waves in a two-layer fluid, when the group
velocity of the surface waves coincides with the phase velocity of the internal waves (see \cite{a21}, \cite{a22}, \cite{a36}.)

In this paper we extend some results obtained in  \cite{DF},  \cite{DFF}, \cite{DFO} to the case where the Schr\"odinger equation is replaced by a cubic complex Ginzburg-Landau (CGL) equation. More precisely, for
\begin{equation}
\label{1.2}
\alpha>0\,\textrm{ and }\,-\frac {\pi}2<\theta<\frac{\pi}2,
\end{equation}
we consider the system
\begin{equation}
\label{1.3}
\left\{\begin{array}{rlll}
e^{-i\theta}u_t&=&u_{xx}-|u|^2u-\alpha g(v)u&(a)\\
\\
v_t+(f(v))_x&=&\alpha (g'(v)|u|^2)_x&(b)
\end{array}\right.
,\qquad x\in\er,\, t	\geq 0,
\end{equation}
for some initial data
\begin{equation}
\label{1.5}
u(0,x)=u_0(x),\quad v(0,x)=v_0(x),
\end{equation}
which can be used to model the interaction between a laser beam and a fluid flow (\cite{Aranson}).

\medbreak

\noindent
The first two sections of the present paper are devouted to the study of the Cauchy problem (\ref{1.3}-\ref{1.5}): in Section \ref{sec:2} we prove the following result concerning local strong solutions:
\begin{Teorema}[\bf Existence and uniqueness of local strong solutions]	\label{t2.1}
	Let $(u_0,v_0)\in H^3(\er)\times H^2(\er)$. Assume that $f\in C^3(\er)$, with $f(0)=0$, and $g$ is of the form
	$$g(v)=\pm v+\rho,\quad \rho\in\er.$$
	Then there exists $T>0$ and a unique strong solution
	$$(u,v)\in C^j([0,T];H^{3-2j}(\er))\times C^j([0,T];H^{2-j}(\er))\quad (j=0;1)$$
	to the Cauchy Problem \eqref{1.3}-\eqref{1.5}.
\end{Teorema}
\noindent
Moreover, in Section 3, we study the global existence of weak solutions to the Cauchy problem for a wider class of flux functions:
\begin{Teorema}[\bf Existence of global weak solutions]
	\label{T4.1}
	Let $(u_0,v_0)\in (H^1(\er))^2$. Assume that $g\in C^3(\er)\cap W^{3,\infty}(\er)$ with $g\geq 0$ and $g^{(3)}$ not affine in any open set. Moreover, assume that $f$ is of the form
	$$f(v)=av^2-bv^3,\quad a,b>0.$$
	Then, there exists a constant $\alpha_0>0$, and, for each $0<\alpha\leq\alpha_0$, functions $u$ and $v$ with 
	\begin{displaymath}
	\begin{array}{lllll}
	u\in L_{loc}^{\infty}([0,+\infty[;H^1(\er))\cap C([0,+\infty[;L^2(\er)),\\
	\\
	u_t\in L_{loc}^2([0,+\infty[;L^2(\er)),\,u(0,\cdot)=u_0(\cdot)
	\end{array}\end{displaymath}
	and
	$$v\in L_{loc}^{\infty}([0,+\infty[;(L^2\cap L^4)(\er))$$
	such that, for all $\phi\in C_0^1(\er\times ]0,+\infty[), \ce)$ and $\psi\in C_0^1(\er\times [0,+\infty[), \er)$,
	\begin{multline*}
	\iint_{\er\times[0,+\infty[}e^{-i\theta}u_t\phi dxdt+\iint_{\er\times[0,+\infty[}u_x\phi_xdxdt\\
	+\iint_{\er\times[0,+\infty[}|u|^2u\phi dxdt+\alpha \iint_{\er\times[0,+\infty[}g(v)u\phi dxdt=0
	\end{multline*}
	and
	\begin{multline*}
	\iint_{\er\times[0,+\infty[}v\psi_t dxdt+\iint_{\er\times[0,+\infty[}f(v)\psi_x dxdt+\int_{\er}v_0(x)\psi(x,0)dx\\
	+\alpha\iint_{\er\times[0,+\infty[}(g'(v)|u|^2)_x\psi dxdt=0.
	\end{multline*}
\end{Teorema}
 
This result will be obtained by applying the $L^p$ version of the compensated compactness method of F. Murat and L. Tartar (cf.\cite{T}) introduced by M.\-E. Schonbek (cf.\cite{MS}) and the vanishing viscosity method to the  approximating system ($\epsilon>0$) with the same initial data
\begin{equation}
\label{1.4}
\left\{\begin{array}{rlll}
e^{-i\theta}u^{\epsilon}_t&=&u^{\epsilon}_{xx}-|u^{\epsilon}|^2u^{\epsilon}-\alpha g(v^{\epsilon})u^{\epsilon}&(a)\\
\\
v^{\epsilon}_t+(f(v^{\epsilon}))_x&=&\alpha (g'(v^{\epsilon})|u^{\epsilon}|^2)_x+\epsilon v^{\epsilon}_{xx}&(b)
\end{array}\right.
,\quad x\in\er,\, t	\geq 0,
\end{equation}

\medbreak

In the second part of the paper (Section \ref{sec:4}), we study the existence of standing wave solutions  for  $ g(v) = v + \rho , \rho > 0 $ (which does not satisfy (6)) and $f(s)=as^2-bs^3$. More precisely, we will look for standing waves of the form $u(t,x)=U(x)$, $V(t,x)=V(x)$, with $U,V$ real solutions of
\begin{equation}
\label{eq:boundstates1}
\begin{cases}
U''-U^3-\alpha (V+\rho)U=0 \\ 
(aV^2-bV^3)'=\alpha (U^2)' 
\end{cases} \qquad \text{ in }  \R
\end{equation}

We denote by $H_{rd}^1(\R)$ the set of functions in $H^1(\R)$ which are even and decreasing in $|x|$.

\begin{Teorema}[\bf Existence of Bound States: focusing case]\label{teo:GS1}
Take $\rho,\alpha>0$. Then:
\begin{enumerate}
\item Assume that $a=0$. Then there exists $b>0$ such that \eqref{eq:boundstates1} admits a solution $(U,V)$, with
\[
U\in H_{rd}^1(\R) \text{ positive}, \qquad V(x)=-\left(\frac{\alpha}{b}\right)^{1/3} U^{2/3}(x).
\]
\item  Assume that $b=0$. Then there exists $a>0$ such that \eqref{eq:boundstates1} admits a solution $(U,V)$, with
\[
U\in H_{rd}^1(\R) \text{ positive}, \qquad V(x)=-\left(\frac{\alpha}{a}\right)^{1/2} U(x).
\]
\item There exists $a,b>0$ such that \eqref{eq:boundstates1} admits a solution $(U,V)$ with $U>0$ and $U\in H^1_{rd}(\R)$, and $V<0$ with $aV^2-bV^3=\alpha U^2$ a.e..
\end{enumerate}
\end{Teorema}
\begin{Rem}\label{rem_V>0} Observe that, since $\int \left((U')^2 +\alpha(V+\rho)U^2 +U^4\right)=0$, there are no solutions $(U,V)$ with $V$ positive and $\rho>0$ in the focusing case.
\end{Rem}
\smallbreak

We obtain stronger results in focusing case  (which cannot be considered in the first part of the paper), that is, we prove existence of real solutions of
\begin{equation}
\label{eq:boundstates3}
\begin{cases}
U''+U^3-\alpha (V+\rho)U=0 \\
(aV^2-bV^3)'=\alpha (U^2)'.
\end{cases}
\end{equation}

\begin{Teorema}[\bf Existence of Bound States: defocusing case]\label{teo:GS2}
Take $\rho>0$.
\begin{enumerate}
\item Let $\alpha>0$ and assume that $a=0$ and $b>0$. Then \eqref{eq:boundstates3} admits a solution $(U,V)$, with
\[
U\in H_{rd}^1(\R) \text{ positive}, \quad V(x)=-\left(\frac{\alpha}{b}\right)^{1/3} U^{2/3}(x).
\]
\item  Let $\alpha>0$ and assume that $a>0$ and $b=0$. Then \eqref{eq:boundstates3} admits two solutions $(U_1,V_1)$ and $(U_2,V_2)$, with
\[
U_1\in H_{rd}^1(\R) \text{ positive},   \qquad V_1(x)=-\left(\frac{\alpha}{a}\right)^{1/2} U_1(x).
\]
and
\[
U_2\in H_{rd}^1(\R) \text{ positive},  \qquad  V_2(x)=\left(\frac{\alpha}{a}\right)^{1/2} U_2(x).
\]
\item Let $a,b>0$. Then, for sufficiently small $\alpha>0$, \eqref{eq:boundstates3} admits two pairs of solutions $(U_1,V_1)$ and $(U_2,V_2)$, with
\[
U_1>0, \ V_1>0 \qquad \text{ and } \qquad U_2>0,\   V_2<0;
\]
$U_i\in H^1(\R)$ and $aV_i^2-bV_i^3=\alpha U_i^2$ for $i=1,2$.
\end{enumerate}
\end{Teorema}

These last two theorems complement some results in \cite{CDW}, \cite{CDP} and \cite{DFO1}.
The techniques involve variation methods for elliptic problems, and consist on either minimization with $L^p $--constraints or minimizations using Nehari-type manifolds.

\bigskip

The CGL equation describes (cf.\cite{AK}) a large class of phenomena like phase transitions, superconductivity, superfluidity and Bose-Einstein condensation to liquid crystals. The coupling of a CGL equation with a quasilinear conservation law can describe the interaction between a laser beam and a fluid flow. Other examples of coupling can be considered (\cite{S} and \cite{TK}).This kind of interactions are particular cases of the general theory of the interactions between short and long waves motivated by the seminal paper of D.J.Benney (\cite{B}) and first studied in the special case of $f$ linear, $g(v)=v$ and $\theta = \frac \pi 2$ (Schr\"odinger equation) by M.Tsutsumi and S.Hatano (cf .\cite{TH1} and \cite{TH2}).
\end{section}
\section{Existence and uniqueness of local strong solutions}\label{sec:2}
\noindent
The main idea to establish Theorem \ref{t2.1} is to apply a variant of T. Kato's Theorem 6 in \cite{K} by introducing a change of the dependent variables $(u,v)$, as done in \cite{DFO} (see also \cite{O} and \cite{ST}).\\ 
Let us put, for $f$ and $g$ verifying the assumptions of Theorem \ref{t2.1},
\begin{equation}
\label{2.1}
F=u_t.
\end{equation}
Equation (\ref{1.3}-a) can then be rewritten as
\begin{equation}
\label{2.2}
u=(\partial_{xx}-1)^{-1}(|u|^2u+u(\alpha g(v)-1)+e^{-i\theta}F).
\end{equation}
Also, by  differentiating (\ref{1.3}-a) with respect to $t$ and using equation (\ref{1.3}-b), we obtain that
\begin{multline*}
F_t-e^{i\theta}F_{xx}\\=-e^{i\theta} \left(2|u|^2F+u^2\overline{F}+\alpha F g(v)-\alpha ug'(v)f'(v)v_x+\alpha^2ug'(v)(g'(v)|u|^2)_x\right).
\end{multline*}
Hence, instead of the Cauchy Problem \eqref{1.3}-\eqref{1.5}, we will consider the following alternative problem, which has the advantage of not presenting derivative losses in the nonlinear term:
\begin{equation}
\label{2.3}
\left\{\begin{array}{rlll}
F_t-e^{i\theta}F_{xx}&=&K(t,F,v)\\
\\
v_t+(f(v))_x&=&\alpha (g'(v)|\tilde{u}|^2)_x,
\end{array}\right.
\end{equation}
with 
\begin{multline*}
K(t,F,v)\\=-e^{i\theta}\left(2|u|^2F+u^2\overline{F}+\alpha F g(v)-\alpha ug'(v)f'(v)v_x+\alpha^2ug'(v)(g'(v)|\tilde{u}|^2)_x \right), 
\end{multline*}
where 
$$u(t,x)=u_0(x)+\int_0^tF(s,x)ds,$$
$$\tilde{u}=(\partial_{xx}-1)^{-1}(|u|^2u+u(\alpha g(v)-1)+e^{-i\theta}F)$$
and for initial data
\begin{equation}\label{2.5}
\begin{split}
F(0,x)&=F_0(x):=e^{i\theta}({u_0}_{xx}(x)-|u_0(x)|^2u_0(x))-\alpha g(v_0(x))u_0(x))   \in H^1(\er)\\
 v(0,x)&=v_0(x)\in H^2(\er).
\end{split}
\end{equation}
Concerning this new problem, we will show the following:
\begin{Lemma}
\label{Lemanovo}
Let $(F_0,v_0)\in H^1(\er)\times H^2(\er)$. Then, there exists $T>0$ and a unique strong solution
$$(F,v)\in C^j([0;T];H^{1-2j}(\er))\times C^j([0;T];H^j(\er)),\quad (j=0;1)$$
to \eqref{2.3} with $F(0,\cdot)=F_0$ and $v(0,\cdot)=v_0$.
\end{Lemma}

\medskip

\noindent
\begin{proof}
We begin by setting this Cauchy Problem in the framework of real spaces. Putting
$$F_1=\Re(F),\,F_2=\Im(F), u_1=\Re(u),\,u_2=\Im(u),$$
$${F_1}_0=\Re(F_0)\textrm{ and }{F_2}_0=\Im(F_0),$$
with $U=(F_1,F_2,v)$, system \eqref{2.3} can be rewritten as
\begin{equation} \label{2.6}
U_t+A(U)U=h(t,U)\\
\end{equation}
for initial data
\begin{equation} \label{2.6id}
U(0,x)=(F_1(0,x), F_2(0,x), v (0,x))=({F_1}_0(x),{F_2}_0(x),v_0(x)),
\end{equation}
where 
\begin{equation}
\label{2.7}
A(U)=\left[\begin{array}{ccc}
-\cos(\theta)\,\partial_{xx}&\sin(\theta)\,\partial_{xx}&0\\
-\sin(\theta)\,\partial_{xx}&-\cos(\theta)\,\partial_{xx}&0\\
0&0&f'(v)\partial_x
\end{array}\right]
\end{equation}
and
\begin{equation}
\label{2.8}
h(t,U)=\left[\begin{array}{ccc}
\Re(K(t,F,v))\\
\Im(K(t,F,v))\\
\alpha(g'(v)|\tilde{u}|^2)_x
\end{array}\right]\textrm{(recall that $g'(v) =\pm 1$)}.
\end{equation}
We now note that the operator $e^{i\theta}\partial_{xx}$, $-\frac{\pi}2<\theta<\frac{\pi}2$, is the infinitesimal generator of an analytic semigroup of contractions $(\mathcal{T}_{\theta}(t))_{t\geq 0}$ in $L^2(\er)$, with domain $H^2(\er)$, verifying the estimates (see p. 248 in \cite{CDW2})
\begin{equation}
\label{2.9}
\|\mathcal{T}(t)\psi\|_r\leq (cos(\theta))^{-\frac 12(1-\frac 1p+\frac 1r)}t^{-\frac 12(\frac 1p-\frac 1r)}\|\psi\|_p,\, \forall t>0,\,1\leq p\leq r\leq +\infty.
\end{equation}
Hence, if we set $X=(H^{-1}(\er))^2\times L^2(\er)$, $Y=(H^{1}(\er))^2\times H^2(\er)$, then 
$$A\,:\,U=(F_1,F_2,v)\in W_R \longrightarrow G(X,1,\beta),$$ 
where $R>0$, $W_R=\{U\in Y\,:\,\|U\|_Y<R\}$ and $G(X,1,\beta)$ denotes the set of all linear operators $D$ in $X$ such that $-D$ generates a $C_0$-semigroup $\{e^{-tD}\}_{t\geq 0}$ with, for all $t\geq 0$,    
$$\|e^{-tD}\|_{\mathcal{L}(X)} \leq e^{\beta t},\quad \beta= \frac 12\|f''(v(x))v'(x)\|_{\infty}\leq \gamma(R),$$
where $\gamma$ is a continuous function.\\
Arguing as in the proof of Lemma 2.1 in \cite{DFO} and by adapting the general Kato's theory for quasilinear systems (\cite{K}), we prove the existence of local strong solutions for the Cauchy Problem \eqref{2.6}--\eqref{2.6id}.\end{proof}

\medskip
\noindent
\begin{proof}[Proof of Theorem \ref{t2.1}]
Let $(u_0,v_0)\in H^3(\er)\times H^1(\er)$. For $F_0\in H^1(\er)$ defined by \eqref{2.5}, we consider the solution $(F,v)$ given by Lemma \ref{Lemanovo}.Then, putting
\begin{equation}
\label{defu}
u(t,x)=u_0(x)+\int_0^t F(s,x)dx,
\end{equation}
we deduce
$$u_{tt}(t,x)=F_t(t,x)=e^{i\theta}F_{xx}+K(t,F,v)=$$
$$=e^{i\theta}F_{xx}-e^{i\theta}(2|u|^2F+u^2\overline{F}+\alpha Fg(v)-\alpha u g'(v)f'(v)v_x+\alpha^2ug'(v)(g'(v)|\tilde{u}|^2)_x$$
$$=e^{i\theta}F_{xx}-e^{i\theta}(2|u|^2F+u^2\overline{F}+\alpha Fg(v)-\alpha u g'(v)f'(v)v_x+\alpha u g'(v)(v_t+(f(v))_x))$$
$$=e^{i\theta}F_{xx}-e^{i\theta}(2|u|^2F+u^2\overline{F}+\alpha Fg(v)+\alpha ug'(v)v_t),$$
hence
$$e^{-i\theta}u_{tt}=F_{xx}-2|u|^2F-u^2F-\alpha Fg(v)-\alpha u\frac{\partial}{\partial t}(gv).$$
From $u_{txx}=F_{xx}$ and $u_t=F$ we derive
$$e^{-i\theta} u_{tt}=(u_{xx})_t-2|u|^2u_t-u^2\overline{u}_t-\alpha u_tg(v)-\alpha u\frac{\partial }{\partial t}(g(v))$$
$$=(u_{xx})_t-\alpha\frac{\partial}{\partial t}(ug(v))-\frac{\partial}{\partial t}(|u|^2u),$$
and we obtained that
$$(e^{-i\theta} u_t-u_{xx}+\alpha g(v)u+|u|^2u)_t=0.$$
and
$$e^{-i\theta} u_t-u_{xx}+\alpha g(v)u+|u|^2u=(e^{-i\theta} u_t-u_{xx}+\alpha g(v)u+|u|^2u)|_{t=0}$$
$$=e^{-i\theta}F(x,0)-u_{xx}(x,0)+\alpha u(x,0)g(v(x,0))+|u|^2u(x,0)=0:$$
We obtained that
$$e^{-i\theta}u_t=u_{xx}-|u|^2u-\alpha ug(v).$$
Noticing that $u_{xx}=e^{-i\theta}u_t+|u|^2u+\alpha ug(v),$
$$u=(\partial_{xx}-1)^{-1}(|u|^2u+u(\alpha g(v)-1)+e^{-i\theta}F)=\tilde{u},$$
and so
$$v_t+(f(v))_x=\alpha(g'(v)|u|^2)_x:$$
we showed that $(u,v)$ is a solution of the Cauchy Problem \eqref{1.3}-\eqref{1.5}. Also, from Lemma \eqref{Lemanovo} and \eqref{defu}, we obtain that $u\in C([0,T];H^3(\er))$.\end{proof}

\section{Existence of a weak solution}\label{sec:3}
\noindent
We begin this section by deriving an {\it a priori} estimate for the solutions of system \ref{1.4}, which extends Lemma 1.2 in \cite{DF} and Lemma 2.2 in \cite{DFF}:
\begin{Prop}
	\label{P3.2}
	Let 
	\[
	(u^{\epsilon},v^{\epsilon})\in C([0,+\infty[;(H^1(\er))^2)\cap W_{loc}^{1,2}([0,+\infty[;(L^2(\er))^2)
	\]
	be a solution of system \ref{1.4} with initial data $({u_0}^{\epsilon},{v_0}^{\epsilon})=(u_0,v_0)\in (H^1(\er))^2$, with $f$ and $g$ verifying the assumptions of Theorem \ref{T4.1}. Then, there exists a constant $\alpha_0>0$ independent of $\epsilon$ and a positive function $h\in C([0,+\infty[)$, independent of $\alpha$ and $\epsilon$ such that for $\alpha\leq \alpha_0$, $\epsilon\leq 1$ and for all $t\geq 0$ we have
	\begin{multline}
	\label{3.15}
	\int |u^{\epsilon}|^2+\int |u_x^{\epsilon}|^2+\int (v^{\epsilon})^2+\int (v^{\epsilon})^4+\int_0^t\int |u^{\epsilon}_x|^2dxd\tau\\
	+\int_0^t\int |u^{\epsilon}_t|^2dxd\tau+\epsilon\int_0^t\int(v_x^{\epsilon})^2dxd\tau\leq h(t).
	\end{multline}
\end{Prop}
\noindent
\begin{proof}
\noindent
For convenience, we drop the superscript $\epsilon$. 
We multiply equation (\ref{1.4}-a) by $\overline{u}_t$ and integrate in $\er$ to obtain, taking the real part and denoting  $\int_{\er} \cdot\, dx$ simply by  $\int \cdot \,$,
$$\displaystyle \cos \theta \int |u_t|^2+\frac 12 \frac d{dt}\int \Big(|u_x|^2+\frac 12|u|^4+\alpha g(v)|u|^2\Big)$$
\begin{displaymath}
\begin{array}{lllll}
&=&\displaystyle\frac{\alpha}2\int g'(v)|u|^2v_t\\
&=&\displaystyle\frac{\alpha}2\int \Big(g'(v)|u|^2\Big[-(f(v))_x+\alpha (g'(v)|u|^2)_x+\epsilon v_{xx}\Big]\Big)\\
&=&\displaystyle\frac{\alpha}2\int \Big(g'(v)|u|^2\Big)_xf(v)-\frac{\alpha}2\epsilon\int g''(v)|u|^2(v_x)^2\\
&&-\displaystyle\frac{\alpha}2\epsilon\int g'(v)(|u|^2)_xv_x\\
&=&\displaystyle\frac 12\int \Big(v_t+(f(v))_x-\epsilon v_{xx}\Big)f(v)-\frac{\alpha}2\epsilon \int g''(v)|u|^2(v_x)^2\\
&&-\displaystyle\frac{\alpha}2\epsilon\int g'(v)(|u|^2)_xv_x,
\end{array}
\end{displaymath}
and so, for $t\in [0,+\infty[$,
\begin{multline}
\label{3.2}
2\cos\theta \int |u_t|^2+\frac{d}{d t}\int\Big(|u_x|^2+\frac 12|u|^4+\alpha g(v)|u|^2-F(v)\Big)\\
-\epsilon \int f'(v)(v_x)^2+\epsilon \alpha \int g'(v)(|u|^2)_xv_x+\epsilon \alpha\int g''(v)|u|^2(v_x)^2=0,
\end{multline}
where 
$$F(v)=\int_0^vf(\xi)d\xi.$$
Now, multiplying (\ref{1.4}-a) by $e^{i\theta}\overline{u}$, integrating in $\er$ and taking the real part, we obtain, for $t\in[0,+\infty[$, 
\begin{equation}
\label{3.3}
\frac d{dt}\int |u|^2+2\cos\theta \int\Big(|u_x|^2+|u|^4+\alpha g(v)|u|^2\Big)=0.
\end{equation}
Finally, multiplying (\ref{1.4}-b) by $v$ and integrating, we obtain
\begin{displaymath}
\begin{array}{llll}
\displaystyle \frac 12\frac{d}{dt}\int v^2&=&\displaystyle\int v\Big(\alpha (g'(v)|u|^2)_x+\epsilon v_{xx}-(f(v))_x\Big)\\
&=&\displaystyle -\epsilon \int (v_x)^2-\alpha \int (g(v))_x|u|^2\\
&=&\displaystyle -\epsilon \int (v_x)^2+2\alpha \int \Re(g(v)u\overline{u}_x)\\
&=&\displaystyle -\epsilon \int (v_x)^2+2\int\Re((u_{xx}-|u|^2u-e^{-i\theta} u_t)\overline{u}_x)\\
&=&\displaystyle -\epsilon \int (v_x)^2-2\Re\Big(\int e^{-i\theta}u_t\overline{u}_x\Big)\\
&=&\displaystyle -\epsilon \int (v_x)^2-2\sin\theta\Im\Big(\int u_t\overline{u}_x\Big)-2\cos\theta\Re\Big(\int u_t\overline{u}_x\Big).
\end{array}
\end{displaymath}
and, since 
$$\frac d{dt}\Im\Big(\int u\overline{u}_x\Big)=2\Im\int\Big(u_t\overline{u}_x\Big).$$
we obtain for $t\in[0,+\infty[$
\begin{equation}
\label{3.4}
\frac d{dt}\Big(\frac 12\int v^2+\sin\theta \Im\Big(\int u\overline{u}_x\Big) \Big)+\epsilon\int (v_x)^2+2\cos\theta \Re\Big(\int u_t\overline{u}_x\Big)=0.
\end{equation}
From \eqref{3.4} we easily derive, for $t\in[0,+\infty[$, 
\begin{multline}
\label{3.5}
\epsilon\int_0^t\int (v_x)^2dxd\tau+\frac 12 \int v^2\leq |\sin\theta| \|u\|_2\|u_x\|_2\\
+2\cos\theta\Big(\int_0^t\|u_x\|_2^2d\tau\Big)^{\frac 12}\Big(\int_0^t\|u_t\|_2^2d\tau\Big)^{\frac 12}+M_0,
\end{multline}
with
\begin{equation}
\label{3.6}
M_0=\frac 12\int v_0^2+\sin\theta \Im\Big(\int u_0{\overline{u}_0}_x\Big).
\end{equation}

Now, recall that
\begin{equation}
\label{3.7}
g(\xi)\geq 0,\quad \xi\in\er.
\end{equation}
Since $\alpha>0$, we derive from \eqref{3.3}, for $t\in[0,+\infty[
$,
\begin{equation}
\label{3.8}
\int |u|^2+2\cos\theta\Big(\int_0^t\int |u_x|^2dxd\tau+\int_0^t\int |u|^4dxd\tau\Big)\leq \int |u_0|^2.
\end{equation}
\begin{Rem} 
\label{R3.1}
If the support of $g'$ is contained in $[0,+\infty[$ it is not difficult, for $v_0\geq 0$ a.e., to deduce, from (\ref{1.4}-b) and for a fixed $\epsilon>0$, that $v(t,x)\geq 0$ a.e. in $[0,+\infty[\times \er$. Indeed, putting $v_{-}:=-\min(v,0)$, 
$$v_tv_{-}=(v_{-})_tv_{-},\quad v_xv_{-}=(v_-)_xv_{-}$$
$$\textrm { and }v_x(v_{-})_x=(({v_{-}})_x)^2\, \textrm{ (cf. \cite{GR}, chap. II).}$$
Multiplying (\ref{1.4}-b) by $v_{-}$ and integrating in space and in the time interval $[0,t]$ yields
\begin{multline*}
\frac 12\int (v_-)^2-\frac 12\int ({v_0}_-)^2+\epsilon\int_0^t\int ((v_{-})_x)^2dxd\tau\\
\leq \frac{\epsilon}2\int_0^t\int((v_-)_x)^2dxd\tau+\frac 1{2\epsilon}\int_0^t\int (f'(v))^2(v_-)^2dxd\tau,
\end{multline*}
from where we deduce that
$$\int(v_-)^2\leq \frac 1{2\epsilon}\int_0^t\int (f'(v))^2(v_-)^2dxd\tau,\, t\geq 0,$$
wich implies, since $f'(v)\in L^{\infty}$, that $v_{-}=0$ a.e. (by Gronwall's inequality). \\
In this case, and for $v_0\geq0$  a.e., we can replace \eqref{3.7} by $g(\xi)\geq 0$ for $\xi\geq 0$.
\end{Rem}

From \eqref{3.5} and \eqref{3.8} we derive, for $t\in[0,+\infty[$,
\begin{equation}
\label{3.9}
\epsilon\int_0^t\int (v_x)^2dxd\tau+\frac 12\int v^2\leq c\Big(1+\|u_x\|_2^2+\int_0^t\|u_t\|_2^2d\tau\Big)^{\frac 12},
\end{equation}
with $c > 0$ independent of $\alpha$ and $\epsilon$.

\medbreak

\noindent
To simplify, we take $a=b=1$, so that $f(v)=v^2-v^3$. We have $F(v)=\frac 13v^3-\frac 14v^4$ and $f'(v)=2v-3v^2$, and so, since $2v\leq 3v^2+\frac 13$, 
\begin{equation}\label{eq:f'(v)leq}
-\epsilon\int f'(v)(v_x)^2dxd\tau\geq -\frac 13\epsilon\int (v_x)^2.
\end{equation}
For positive constants $c_0$ and $c_1$,
\begin{equation}\label{eq:F(v)leq}
-\int F(v)\geq -c_0\int v^2+c_1\int v^4.
\end{equation}
Moreover, by integrating in $[0,t]$ equation (18), and using \eqref{3.7}, \eqref{eq:f'(v)leq} and \eqref{eq:F(v)leq}, we  deduce that
\begin{multline*}
2\cos \theta \int_0^t \int |u_t|^2+\int |u_x|^2 + \frac{1}{2}\int|u|^4+c_1\int v^4-\frac{\epsilon}{3}\int_0^t \int(v_x)^2 \\
\leq \int (|(u_0)_x|^2+\frac{1}{2}|u_0|^4+\alpha g(v_0)|u_0|2-F(v_0))\\ -\epsilon \alpha \int_0^t \int g'(v) (|u|^2)_xv_x -\epsilon \alpha \int_0^t \int g''(v) |u|^2 (v_x)^2.
\end{multline*}
Combining this with  \eqref{3.8} and \eqref{3.9}  yields, for $t\in[0,+\infty[$, 
\begin{multline}\label{3.10}
\int |u|^2+\int |u|^4+\int v^2+\int v^4+\int_0^t\int |u_x|^2dxd\tau$$
\\+\int_0^t\int |u|^4dxd\tau+\int |u_x|^2+\int_0^t\int |u_t|^2dxd\tau+\epsilon\int_0^t\int (v_x)^2dxd\tau\\
\leq c\Big(1+\alpha\epsilon\int_0^t\int|uu_xv_x|dxd\tau+\alpha\epsilon\int_0^t\int |u|^2(v_x)^2dxd\tau\Big)\\
+c\Big(1+\|u_x\|_2^2+\int_0^t\|u_t\|_2^2d\tau\Big)^{\frac 12}.
\end{multline}
where $c>0$ is a constant independent of $\alpha\leq \alpha_0$ (for some $\alpha_0$) and $\epsilon$.
Let us set
\begin{equation}
\label{3.11}
q(t)=1+\|u_x\|_2^2+\int_0^t\|u_t\|_2^2d\tau.
\end{equation}
We deduce from \eqref{3.10} and the Gagliardo-Nirenberg inequality $$\|u\|_{\infty}\leq \|u\|_2^{\frac 12}\|u_x\|_2^{\frac 12} \leq \|u_0\|_2^{\frac 12}\|u_x\|_2^{\frac 12}$$ (recall \eqref{3.3}) that, for $u\in H^1(\er)$,  
\begin{multline}
\label{3.12}
q(t)+\epsilon\int_0^t\int(v_x)^2dxd\tau\leq c\big(1+\alpha\epsilon\int_0^t\|u_x\|_2^{\frac 32}\|v_x\|_2d\tau\\
+\alpha\epsilon\int_0^t\|u_x\|_2\|v_x\|_2^2d\tau+q^{\frac 12}(t)\Big),
\end{multline}
hence
\begin{equation}
\label{3.13}
q(t)\leq \psi(t):=\kappa\big(1+\alpha\epsilon\int_0^t\|u_x\|_2^{\frac 32}\|v_x\|_2d\tau+\alpha\epsilon\int_0^t\|u_x\|_2\|v_x\|_2^2d\tau\Big),
\end{equation}
Now,
$$\psi'(t)=\kappa \alpha\epsilon\Big(\|u_x\|_2^{\frac 32}\|v_x\|_2+\|u_x\|_2\|v_x\|_2^2\Big)\leq \kappa \alpha\epsilon\Big(\psi^{\frac 34}(t)\|v_x\|_2+\psi^{\frac 12}(t)\|v_x\|_2^2\Big),$$
and
$$\psi'(t)\psi^{-\frac 12}(t)\leq \kappa\alpha\epsilon\Big(\psi^{\frac 14}(t)\|v_x\|_2+\|v_x\|_2^2\Big),$$
and
\begin{displaymath}
\begin{array}{llll}
\displaystyle-2\int_0^t\psi^{\frac 12}(\tau)e^{\tau}d\tau+2\Big[\psi^{\frac 12}(\tau)e^{\tau}\Big]_0^{\tau}   &=&   \displaystyle\int_0^t\psi'(\tau)\theta^{-\frac 12}(\tau)e^{\tau}d\tau\\
&&\\
&\leq&\displaystyle \kappa \Big[\alpha \epsilon^{\frac 12}\psi^{\frac 14}(t)\Big(\int_0^te^{2\tau}d\tau\Big)^{\frac 12}\Big(\int_0^t\epsilon\|v_x\|_2^2d\tau\Big)^{\frac 12}\\
\\
&&+\alpha e^t\int_0^t\epsilon\|v_x\|_2^2d\tau\Big]\\
\\
&\leq& \kappa \alpha \epsilon^{\frac 12}\psi^{\frac 12}(t)e^t+\kappa \alpha \epsilon e^t\psi^{\frac 12}(t),
\end{array}
\end{displaymath}
therefore
$$\displaystyle-2\int_0^t\psi^{\frac 12}(\tau)e^{\tau}d\tau+2\Big[\psi^{\frac 12}(\tau)e^{\tau}\Big]_0^{\tau} \leq  \kappa \alpha \epsilon^{\frac 12}\psi^{\frac 12}(t)e^t+\kappa \alpha \epsilon e^t\psi^{\frac 12}(t),$$
by \eqref{3.9} and \eqref{3.13}. Hence, for $\epsilon\leq 1$,
$$
\phi^{\frac 12}(t)e^t\leq c\alpha\phi^{\frac 12}(t)e^t+c\int_0^t\phi^{\frac 12}e^{\tau}d\tau+c,
$$
and, for $\alpha\leq \alpha_0$ such that $1-c\alpha>0$, we derive, by Gronwall's inequality,
\begin{equation}
\label{3.14}
\phi(t)\leq h_1(t),\,h_1\in C(\er^+),
\end{equation}
hence, in view of \eqref{3.13},
$$q(t)=1+\|u_x\|_2^2+\int_0^t\|u_t\|_2^2d\tau\leq \phi(t)\leq h_1(t).$$
Finally, combining this inequality with \eqref{3.8}, \eqref{3.9} and \eqref{3.10}, we deduce the inequality stated in Proposition \ref{P3.2}.
\end{proof}

\bigskip

\noindent
Next, this {\it a priori} estimate will allow us to  show the existence of a global unique strong solution to the approximated system  \eqref{1.4}:
\begin{Teorema}
	\label{Tee}
	Let $f$ and $g$ verifying the assumptions of Theorem \ref{T4.1} and $(u_0,v_0)\in (H^1(\er))^2$. Then there exists a unique solution
	\[
	(u^{\epsilon},v^{\epsilon})\in C([0,+\infty[;(H^1(\er))^2)\cap W_{loc}^{1,2}([0,+\infty[;(L^2(\er))^2)
	\]
	to system \ref{1.4} with initial data $({u_0}^{\epsilon},{v_0}^{\epsilon})=(u_0,v_0)\in (H^1(\er))^2$.
\end{Teorema}
\noindent
\begin{proof}
\noindent
Once again, we drop the superscript $\epsilon$.\\
For the (local) existence in $C([0,T'];(H^1(\er))^2)$, for a certain $T' < \infty $, since $e^{i\theta}\partial_{xx}$, $-\frac {\pi}2<\theta<\frac{\pi}2$, is the infinitesimal generator of an analytic semigroup of contractions $\{\mathcal{T}_\theta(t)\}_{t\geq 0}$ in $L^2(\er)$, with domain $H^2(\er)$ and verifying the estimate \eqref{2.9}, the result follows by considering the Duhamel formulas of the auxiliary system 
\begin{displaymath}
\left\{\begin{array}{rlll}
e^{-i\theta}u_t-u_{xx}&=&-|\tilde{u}|^2\tilde{u}-\alpha g(\tilde{v})\tilde{u}\\
\\
v_t-v_{xx}&=&-(f(\tilde{v}))_x+\alpha (g'(\tilde{v})|\tilde{u}|^2)_x
\end{array}\right.
\end{displaymath}
and a convenient Banach fixed-point technique (cf. \cite{DF}, Proposition 2.1, for the Schr\"odinger equation case).\\
To conclude that the solution $(u,v)$ is in $W^{1,2}(0,T',(L^2(\er))^2)$ we notice that, since
$$\frac{\partial}{\partial x}\Big(\alpha g'(v)|u|^2-f(v)\Big)\in L^2(0,T',L^2(\er)),$$
$$v\in L^2(0,T';H^2(\er))$$ and $$\frac{\partial v}{\partial t}\in L^2(0,T';L^2(\er)).$$
Moreover, since
$$|u|^2u+\alpha g(v)u,\,\frac{\partial}{\partial x}\Big(|u|^2u+\alpha g(v)u\Big)\in L^2(0,T';L^2(\er)),$$
we also conclude, by the properties of the semigroup $\{\mathcal{T}_\theta(t)\}_{t\geq 0}$, that
$$u\in L^2(0,T';H^2(\er))$$ and $$ \frac{\partial u}{\partial t}\in L^2(0,T';L^2(\er)).$$
To obtain the estimates which yield the global in time existence of $(u,v)$, we apply \eqref{3.15} and the Duhamel formula for the heat equation
$$v(t)=e^{\Delta t}v_0+\int_0^te^{\Delta(t-\tau)}\frac{\partial}{\partial x}\Big[\alpha g'(v)|u|^2-f(v)\Big](\tau)d\tau$$
and the well-known estimate
$$\Big\|\frac{\partial}{\partial x}e^{\Delta t}\psi\Big\|_2\leq \frac{C}{t^{\frac 34}}\|\psi\|_1,\quad t>0$$
to obtain
$$\|v_x(t)\|_2\leq \|v_{0x}\|_2+c\int_0^t\frac{1}{(t-\tau)^{\frac 34}}\Big[\|u\|_2\|v\|_2+(\|u\|_4^2+\|v\|_4^2+\|v\|_2)\|v_x\|_2\Big](\tau)d\tau$$
$$\leq \|{v_0}_x\|_2+c\int_0^t\frac{d\tau}{(t-\tau)^{\frac 34}}+c\int_0^t\|v_x(\tau)\|_2\frac{d\tau}{(t-\tau)^{\frac 34}}$$
by \eqref{3.15}. We then obtain an estimate for
$\int_0^t\|v_{xx}(\tau)\|_2^2$ and for $\int_0^t\|v_{t}(\tau)\|_2^2$, which achieves the proof.
\end{proof}

\bigskip

\noindent
We are now in condition to show the existence of a global weak solution to the initial Ginzburg Landau System:

\medskip

\noindent
\begin{proof}[Proof of Theorem \ref{T4.1}]
Let  $(u_0,v_0)\in (H^1(\er))^2$ and 
$$(u^{\epsilon},v^{\epsilon})\in C([0,+\infty[;(H^1(\er))^2\cap W_{loc}^{1,2}([0,+\infty[;(L^2(\er)^2)$$ the corresponding solution of \eqref{1.4} for initial data $(u_0,v_0)$. Furthermore, we assume that $\epsilon\leq 1$ and $\alpha\leq \alpha_0$ (see Proposition \ref{P3.2}).\\

\noindent
Let us fix $T>0$. We have, by \eqref{3.15},  that $(u^{\epsilon})_{\epsilon}$ is bounded in $$L^{\infty}([0,T[;H^1(\er))$$ and $(u^{\epsilon}_t)_{\epsilon}$ is bounded in $$L^{2}([0,T[;L^2(\er))\cap L^{\infty}([0,T[;H^{-1}(\er)).$$
By applying Aubin's Lemma (for each interval $]-R,R[\times]-T,T[$) and a standard diagonal extraction, there exists $u\in L^{\infty}([0,T[;H^1(\er))$ such that,  $u_t\in L^2([0,T[;L^2(\er))$, for each $T>0$, and a subsequence still denoted $(u^{\epsilon})_{\epsilon}$ such that
\begin{equation}
\label{4.1}
\begin{array}{llll}
u^{\epsilon}\rightharpoonup u\textrm{ in }L^{\infty}([0,T[;H^1(\er)) \textrm{weakly-$\ast$ and a.e. in } [0,+\infty[\times \er,\\ 
u_t^{\epsilon}\rightharpoonup u_t\textrm{ in }L^{2}([0,T[;L^2(\er)) \textrm{weakly}.
\end{array}
\end{equation}
Hence, $u\in C([0,+\infty[;L^2(\er))$ and $u(0)=u_0$.\\
We can also deduce, by \eqref{3.15}, that there exists $v\in L^{\infty}([0,T[;(L^4\cap L^2)(\er))$ and $w\in L^{\infty}([0,T[;L^{\frac 43}(\er))$ such that 
\begin{equation}
\label{4.2}
\begin{array}{llll}
v^{\epsilon}\rightharpoonup v\textrm{ in }L^{\infty}([0,T[;L^4\cap L^2(\er)) \textrm{weakly-$\ast$},\\ 
f(v^{\epsilon})-\alpha|u^{\epsilon}|^2g'(v^{\epsilon})\rightharpoonup w\textrm{ in }L^{\infty}([0,T[;L^{\frac 43}(\er)) \textrm{weakly-$\ast$}.
\end{array}
\end{equation}
Moreover we have, with $p=3$,
$$|f'(\xi)|\leq c(1+|\xi|^{p-1}),\,\xi\in\er$$
and for each real convex $C^2$ entropy function $\eta$ with compact support, we deduce from \eqref{1.4} with $q_1$ and $q_2 \in C^2(\er)$ such that $q_1'=\eta'f'$  and  $q_2'=\alpha\eta'g''$, 
$$\frac{\partial}{\partial t}\eta(v^{\epsilon})+\frac{\partial}{\partial x}(q_1(v^{\epsilon})- |u^{\epsilon}|^2 q_2(v^{\epsilon})) =(\alpha\eta'(v^{\epsilon}) g'(v^{\epsilon})-q_2(v^{\epsilon}))(|u^{\epsilon}|^2)_x+\epsilon\eta'(v^{\epsilon})\frac{\partial^2v^{\epsilon}}{\partial x^2}$$ 
$$=(\alpha\eta'(v^{\epsilon}) g'(v^{\epsilon})-q_2(v^{\epsilon}))(|u^{\epsilon}|^2)_x+\epsilon\frac{\partial^2}{\partial x^2}\eta(v^{\epsilon})-\epsilon\eta''(v^{\epsilon})\Big(\frac{\partial v^{\epsilon}}{\partial x}\Big)^2.$$
Hence, by \eqref{3.15}, we derive (see \cite{DFF}, theorem 2.1, for a similar deduction and argument) that if $\Omega$ is an open bounded subset of $]0,+\infty[\times \er$,
$$\frac{\partial}{\partial t}\eta(v^{\epsilon})+\frac{\partial}{\partial x}(q_1(v^{\epsilon})- |u^{\epsilon}|^2 q_2(v^{\epsilon}))\in K(\Omega)+B(\Omega),$$
where $K(\Omega)$ is a compact set of $H^{-1}(\Omega)$ and $B(\Omega)$ is a bounded set of finite measures in $\Omega$.\\
Because of (32) (and the $ L^2$ strong convergence of $(u^{\epsilon})_{\epsilon}$ in each interval $]-R,R[\times]-T,T[$),  we can now apply a variant of the Corollary 3.1 of Theorem 3.2 in \cite{MS} and, by a suitable diagonal extraction, we can deduce that
$$f(v^{\epsilon})-\alpha|u^{\epsilon}|^2g'(v^{\epsilon})\rightharpoonup f(v)-\alpha|u|^2g'(v)\textrm { in }\mathcal{D}'(]0,+\infty[\times\er).$$
Hence, by \eqref{4.2}, $$f(v^{\epsilon})-\alpha|u^{\epsilon}|^2g'(v^{\epsilon})\rightharpoonup f(v)-\alpha|u|^2g'(v)$$ in $L^{\infty}([0,+\infty[;L^{\frac 43}(\er))\textrm {weakly-$\ast$}$.
\end{proof}

\bigskip

\begin{Rem}
In the framework of Theorem \ref{T4.1}, if $v_0\geq 0$ a.e., we can apply Remark \ref{R3.1} (if the support of $g'$ is contained in $[0,+\infty[$ ) and, since $v^{\epsilon}\rightharpoonup v$ in $L^{\infty}([0,T[;L^4\cap L^2(\er))$ weakly-$\ast$,for each $T>0$, we can conclude, since $v^{\epsilon}\geq 0$ a.e. in $\er\times[0,+\infty[$, that 
$$v\geq 0\textrm{ a.e. in }\er\times[0,+\infty[.$$
Moreover, in this case, we can replace the condition $g\geq 0$ in $\er$ by the weaker condition $g\geq 0$ in $[0,+\infty[$.
\end{Rem}


\section{Existence of standing waves}\label{sec:4}

In this section we will study, by different techniques, the existence of bound states (more precisely, of standing waves) for \eqref{1.3} in the defocusing ($-|u|^2u$) and focusing ($+|u|^2u$) cases,  for $\alpha>0$, $g(v)=v+\rho$ with $\rho>0$,  $f(s)=as^2-bs^3$ with $a,b\geq 0$. In the focusing case, we will consider a minimization problem with an $L^p$ constraint. In the defocusing case, the special structure of the action functional will allow us to consider a minimization problem on a Nehari manifold (thus, solutions will actually be ground-states). 

Observe that the embedding $H^1_r(\R)\hookrightarrow L^q(\R)$ ($q>2$) is not compact, where $H^1_r$ denotes the space of radially symmetric functions (i.e. even functions) of $H^1(\R)$. However, if $u\in L^2(\R)$ is even and decreasing in $|x|$, it is easy to establish 
\[
|u(x)|\leq |x|^{-\frac{1}{2}}\|u\|_{L^2(R)}
\]
(see \cite[p. 341]{Lions1}). Hence, by Strauss' compactness lemma \cite{Strauss}, 
\begin{equation}\label{compactemb}
H^1_{rd}(\R)\hookrightarrow L^q(\R) \ (q> 2),\qquad \text{ is a compact embedding,}  
\end{equation}
 where $H^1_{rd}(\R)=\{u\in H^1_r(\R): u \text{ is decreasing with respect to } |x|\}$.

\subsection{The defocusing case: Proof of Theorem \ref{teo:GS1}}\label{sec:decofusing}

We will look for $(U,V)$ real solution of \eqref{eq:boundstates1} with $U\in H^1(\R)$. We solve the following (equivalent) problem, where a differential equation is coupled with a pointwise identity:
\begin{equation}
\label{eq:boundstates2}
\begin{cases}
U''-U^3-\alpha (V+\rho)U=0 & \text{ in } \R  \\
aV^2-bV^3=\alpha U^2        &\text{ in } \R \\
U(x),\ V(x)\to 0\ \text{ as } |x|\to \infty.
\end{cases}
\end{equation}
We split the proof of Theorem \ref{teo:GS1} in three cases: $a=0$ and $b>0$, $a>0$ and $b=0$, and $a,b>0$. 

\paragraph{Existence for $\mathbf{a=0}$ and some $\mathbf{b>0}$ (Theorem \ref{teo:GS1}-1.)}

In this situation, the second equation in \eqref{eq:boundstates2} is equivalent to
\[
V(x)=-\left(\frac{\alpha}{b}\right)^{1/3} U^{2/3}(x),
\]
hence we aim at solving
\begin{equation}\label{eq:singleequation1}
-U'' + \alpha \rho U -\frac{\alpha^{4/3}}{b^{1/3}} U^{5/3}+ U^3=0 \quad \text{ in } \R.
\end{equation}
Consider the $C^1$ functional $J:H^1(\R)\to \R$ defined as
\begin{equation}\label{eq:functionalJ}
J(U)=\frac{1}{2}\int_{\R} ((U')^2 + \alpha \rho  U^2) + \frac{1}{4}\int U^4,
\end{equation}
constrained to the manifold
\[
\mathcal{M}_1=\left\{U\in H^1(\R):\ \int |U|^{8/3}=1\right\}.
\]

Let us check that $\inf_{\mathcal{M}_1} J$ is achieved. In fact, $J\geq 0$, and we can take a minimizing sequence $(U_n)\subset H^1(\R)$:
\[
U_n\in \mathcal{M}_1,\qquad J(U_n)\to \inf_{\mathcal{M}_1}J.
\] 
By eventually replacing $U_n$ by $|U_n|^*$, the  Schwarz symmetrization of its absolute value, we can assume that $U_n\in H^1_{rd}(\R)$ and $U_n\geq 0$. Moreover,
\[
\frac{1}{2}\min\{1,\alpha \rho\}\|U_n\|_{H^1}^2\leq J(U_n) \leq C 
\]
and, since $\rho,\alpha>0$, $(U_n)_n$ is bounded in $H^1$--norm. Thus there exists $U\in H^1_{rd}(\R)$ such that, up to a subsequence,
$$
\begin{array}{l}
U_n \rightharpoonup U \text{ weakly in $H^1(\R)$},\\
U_n \to U \text{ strongly in $L^p(\R),\ \forall p> 2$}.
\end{array}
$$
(taking into account the compact embedding \eqref{compactemb}). So $U\geq 0$ and $U\in \mathcal{M}_1$;  since $J$ is lower-semicontinuous, $J(U)=\min_{\mathcal{M}_1} J$. Thus, there exists $\lambda\in \R$ (a Lagrange multiplier) such that
\begin{equation}\label{eq:singleequation1}
-U'' + \alpha \rho U -\lambda U^{5/3}+ U^3=0.
\end{equation}
Since $U\not\equiv 0$ ($U\in \mathcal{M}_1$) and $U\geq 0$, by the strong maximum principle we have $U>0$ in $\R$. Testing \eqref{eq:singleequation1} by $U$ itself, we have
\[
\lambda=\int ((U')^2+\alpha \rho U^2 + U^4)>0.
\] Therefore, we can choose $b>0$ in such a way that $\frac{\alpha^{4/3}}{b^{1/3}}=\lambda$. Defining $V(x):=-\left(\frac{\alpha}{b}\right)^{1/3} U^{2/3}$ we have that $(U,V)$ solves  \eqref{eq:boundstates2}, hence is a solution to \eqref{eq:boundstates1}. This proves Theorem \ref{teo:GS1}-1.

\paragraph{Existence for some $\mathbf{a>0}$ and $\mathbf{b=0}$ (Theorem \ref{teo:GS1}-2.)}

This case is very similar to the previous one, hence we just stress the diferences.  Since there are no solutions with $V>0$ (recall Remark \ref{rem_V>0}), we are lead to take 
\begin{equation}\label{eq:equation_auxili2}
V(x)=-\left(\frac{\alpha}{a}\right)^{1/2}U(x),
\end{equation}
and solve
\[
-U'' + \alpha \rho U -\frac{\alpha^{3/2}}{a^{1/2}}U^2 +U^3=0
\]
by minimizing the functional $J$ defined in \eqref{eq:functionalJ}, this time on the manifold
\[
\mathcal{M}_2=\left\{U\in H^1(\R):\ \int |U|^{3}=1\right\}.
\]
The constrained minimization problem leads to the existence of $U\in H^1_{rd}(\R)$ positive solution 
\begin{equation}\label{eq:equation_auxili}
-U'' + \alpha \rho U -\lambda  U^{2}+ U^3=0
\end{equation}
for some $\lambda>0$. Since $\alpha>0$ is fixed, we can choose $a>0$ in such a way that $\frac{\alpha^{3/2}}{a^{1/2}}=\lambda$. Therefore $(U,V)$, with $U$ solution of \eqref{eq:equation_auxili} and $V$ solution of \eqref{eq:equation_auxili2}, solves \eqref{eq:boundstates2}. This proves Theorem \ref{teo:GS1}-2.

\paragraph{Existence for some $\mathbf{a,b>0}$ (Theorem \ref{teo:GS1}-3.)}

This is the most challenging case. Observe that the polynomial $P(s):=s^2- s^3$ vanishes only at $s=0$ and $s=1$, being negative if and only if $s>1$. Moreover, $P$ achieves a local maximum at $P\left(\frac{2}{3}\right)=\frac{4}{27}$. We will obtain a solution $(U,V)$ with $U>0$ and $V<0$.

Consider the restriction $\tilde P:= P|_{]-\infty,0[}$ (which is invertible), and the continuous function
\[
g(s)=\begin{cases}
\tilde P^{-1}(s) & s>0\\
0	& s\leq 0
\end{cases}
\]
which is negative for $s>0$. Observe that, if $V:=g(\alpha U^2)$, then $V^2-V^3=\alpha U^2$. Moreover, 
\begin{equation}\label{eq:lim_g(s)}
\lim_{s\to 0^+} \frac{g(s)}{s^{1/2}}= \lim_{s\to +\infty} \frac{g(s)}{s^{1/3}}=  -1.
\end{equation}
We aim at solving the equation:
\[
-U''+ \alpha (g(\alpha U^2) +\rho) U +U^3=0
\]
and we will succeed to do it, up to a Lagrange multiplier. Define $G(s):=\int_0^s g(\xi)\, d\xi$, which is negative for $s>0$ and  satisfies
\begin{equation}\label{eq:asymptotics_of_G}
\lim_{s\to 0^+} \frac{G(s)}{s^{3/2}}= -\frac{2}{3},\qquad  \lim_{s\to +\infty} \frac{G(s)}{s^{4/3}}= - \frac{3}{4}.
\end{equation}
\begin{Lemma}\label{lemma:minimization_M3}
The following minimization problem has a nonnegative solution:
\[
\inf_{U\in \mathcal{M}_3} J(U), \quad \text{ with } \quad \mathcal{M}_3=\left\{U\in H^1(\R):\ \int G(\alpha U^2)=-1 \right\}.
\]
\end{Lemma}
\begin{proof} 1) We start by checking that $\mathcal{M}_3\neq \emptyset$. Fix $w\in H^1 (\R)$ a positive, radially decreasing function, and take:
 \[
 \varphi(t):=\int_\R G(\alpha t w^2).
 \] Since $\varphi$ is continuous and $\varphi(0)=0$, the claim follows if we prove that $\varphi(t)\to -\infty$ as $t\to +\infty$. From \eqref{eq:asymptotics_of_G}, there exists $A>0$ such that
\begin{equation*}
G(s)\leq -A s^{4/3} \qquad \forall s\geq \alpha w^2(1).
\end{equation*}
Thus, for $x\in [-1,1]$ and $t\geq 1$, $\alpha t w^2(x) \geq \alpha w^2(1)$ and so $G(\alpha t w^2(x))\leq -A \alpha^{4/3} t^{4/3}w^{8/3}(x)$. Therefore, since $G\leq 0$,
\[
\varphi(t)\leq \int_{-1}^1 G(\alpha t w^2(x))\leq -A t^{4/3}\int_{-1}^1 \alpha^{4/3} w^{8/3}(x) \to -\infty,
\]
as $t\to +\infty$.
\smallbreak

\noindent 2) Reasoning exactly as in the proof of the case $a=0$, $b>0$, we can take a minimizing sequence of nonnegative, radially decreasing functions: $(U_n)\in  H^1_{rd}(\R)$ such that $U_n\in \mathcal{M}_3$ and $J(U_n)\to \inf_{\mathcal{M}_3}J$. This sequence is bounded in $H^1(\R)$, thus there exists $U\in H^1_{rd}(\R)$, nonnegative, such that, up to a subsequence, $U_n\to U$ weakly in $H^1(\R)$, strongly in $L^p(\R)$, for $p>2$. From the strong convergence in $L^3(\R)$ and $L^{8/3}(\R)$, there exist $h_1\in L^3(\R)$, $h_2\in L^{8/3}(\R)$ such that 
$|U_n|\leq h_1,h_2$ for every   $n\in \N$. Moreover, since $|G(s)|\leq C (|s|^{3/2}+|s|^{4/3})$, then 
\[
|G(\alpha U_n^2)|\leq C' (|U_n|^3+|U_n|^{8/3})\leq C'(h_1^3+h_2^{8/3})\in L^1(\R),
\] and by Lebesgue's Dominated Convergence Theorem we have
\[
-1=\lim \int G(\alpha U_n^2)=\int G(\alpha U^2).
\]
Therefore the proof of this lemma is complete, by observing that
\[
\inf_{\mathcal{M}_3} J\leq J(U) \leq \liminf J(U_n) = \inf_{\mathcal{M}_3} J. \qedhere
\]
\end{proof}

After the previous lemma, we are ready to prove the  existence result for some $a,b>0$.

\medbreak

\begin{proof}[Proof of Theorem \ref{eq:boundstates1}-3]
The previous lemma yields the existence of a nontrivial $H^1(\R)$--solution (which is nonnegative and radially decreasing) to the problem
\[
-U''+ \alpha (\lambda g(\alpha U^2)+\rho)U+U^3=0 \text{ in } \R,
\]
for some $\lambda\in \R$. The strong maximum principle yields that $U>0$. By testing the equation by $U$ itself, we see that
\[
\int_\R ((U')^2 +\alpha \rho U^2+U^4) + \lambda \int_\R \alpha g(\alpha U^2)U^2 =0,
\]
and since $g(\alpha U^2)<0$, then actually $\lambda>0$.  Take $V(x):=\lambda g(\alpha U^2)<0$. Then, by definition of $g$,
\[
\frac{1}{\lambda^2}V^2-\frac{1}{\lambda^3} V^3=P(\frac{1}{\lambda} V)=\alpha U^2.
\]
In conclusion, $(U,V)$ solves \eqref{eq:boundstates3} with the choice $a:=\frac{1}{\lambda^2}>0$, $b=\frac{1}{\lambda^3}>0$.
\end{proof}
\subsection{The focusing case: Proof of Theorem \ref{teo:GS2}}

We will now look for $(U,V)$ real solutions of \eqref{eq:boundstates3}, solving instead the following (equivalent) problem:
\begin{equation}
\label{eq:boundstates4}
\begin{cases}
U''+U^3-\alpha (V+\rho)U=0 & \text{ in } \R  \\
aV^2-bV^3=\alpha U^2        &\text{ in } \R \\
U(x),\ V(x)\to 0\ \text{ as } |x|\to \infty.
\end{cases}
\end{equation}

The results in this case are more complete, since we can use a Nehari manifold/Mountain pass approach instead of an $L^p$ constraint. As in the defocusing case, we split the proof of Theorem \ref{teo:GS2} in three cases.

\paragraph{Existence for $\mathbf{a=0}$ and $\mathbf{b>0}$ (Theorem \ref{teo:GS2}-1.)}

In this situation, the second equation in \eqref{eq:boundstates4} is equivalent to
\[
V(x)=-\left(\frac{\alpha}{b}\right)^{1/3} U^{2/3}(x),
\]
hence we aim at solving
\begin{equation}\label{eq:singleequation1}
-U'' + \alpha \rho U =\frac{\alpha^{4/3}}{b^{1/3}} U^{5/3}+ U^3.
\end{equation}

Weak solutions of \eqref{eq:singleequation1} correspond to critical points of the $C^1$-action functional $A:H^1(\R)\to \R$ defined by
\[
A(U)=\frac{1}{2}\int \left((U')^2 + \rho \alpha U^2\right) - \frac{3\alpha^{4/3}}{8b^{1/3}}\int |U|^{8/3} - \frac{1}{4}\int U^4.
\]  
We introduce the Nehari set
\begin{align*}
\mathcal{N}&=\{U\in H^1(\R):\ U\neq 0,\ A'(U)[U]=0\}\\
			& =\left\{U\in H^1(\R):\ U\neq 0,\ \int ((U')^2+\rho\alpha U^2)=\frac{\alpha^{4/3}}{b^{1/3}}\int |U|^{8/3}+\int U^4\right\}.
\end{align*}

\begin{proof}[Proof of Theorem \ref{teo:GS2}-1]
We will prove this result by showing that the quantity
\[
\inf_{U\in \mathcal{N}} A(U)
\]
is a critical level of $A$, being achieved by a positive solution of \eqref{eq:singleequation1}.
Although this fact follows from standard arguments, we sketch the proof since we are dealing with an unbounded set - $\R$ - and working with one space dimension. Since the proof is long, we split it in several steps.

\smallbreak

\noindent 1) Given $U\in H^1(\R)\setminus \{0\}$, let us check the existence of $t>0$ such that $tU \in \mathcal{N}$. Consider the map $\varphi_U:[0,+\infty[\to \R$ defined as
\[
\varphi_U(t):=A(tU)=\frac{t^2}{2}\int((U')^2+\rho \alpha U^2)-t^{8/3} \frac{3\alpha^{4/3}}{8 b^{1/3}}\int |U|^{8/3}-\frac{t^4}{4} \int U^4.
\]
We have $\varphi(0)=0$, $\varphi(t)\to -\infty$ as $t\to +\infty$, and $\varphi(t)>0$ for $t>0$ sufficiently small. Then $\varphi_U(t)$ admits a critical point at $t^*>0$, which corresponds to a point $t^*U\in \mathcal{N}$ (it is actually simple to see that $t^*$ is the unique positive critical point of $\varphi_U$, corresponding to its global maximum). An important observation that we use ahead is that, if in addition $A'(U)[U]\leq 0$, then $\varphi_U'(1)\leq 0$, and so $t^*\leq 1$.

\smallbreak

\noindent 2) The set $\mathcal{N}$ is a $C^1$-manifold. In fact, for $F(U):=A'(U)[U]$ with $U\in \mathcal{N}$, we have
\[
\begin{split}
F'(U)[U] &= 2\int((U')^2+\rho \alpha U^2)-\frac{8\alpha^{4/3}}{3b^{1/3}}\int |U|^{8/3}-4\int U^4\\
		&= -\frac{2}{3}\frac{\alpha^{4/3}}{b^{1/3}}\int |U|^{8/3}-2 \int |U|^4<0.
\end{split}
\]
Moreover, this implies that constrained critical points are free critical points: for $U\in \mathcal{N}$ such that $A|_\mathcal{N}'(U)=0$, there exists $\lambda\in \R$, a Lagrange multiplier, such that $A'(U)=\lambda F'(U)$. Using $U$ as test function, we see that $\lambda=0$, thus $A'(U)=0$.
\smallbreak

\noindent 3) Combining the Sobolev embeddings $H^1(\R)\hookrightarrow L^{8/3}(\R),L^4(\R)$ with the definition of $\mathcal{N}$, we deduce the existence of  $C_1,C_2>0$ such that
\[
 (\|U\|_{8/3}^2+\|U\|_4^2) \leq  C_1 \|U\|_{H^1}^2 \leq C_2 (\|U\|^{8/3}_{8/3}+\|U\|_4^4) \quad \forall U\in \mathcal{N}.
\]
Since $2<8/3<4$,  there exists $\delta>0$ such that
\begin{equation}\label{eq:N_bded_below}
\|U\|_{8/3}+\|U\|_4 \geq \delta \qquad \forall U\in H^1(\R).
\end{equation}
\noindent 4) For $U\in \mathcal{N}$, we have
\begin{equation}\label{eq:A(U)_alternatize}
A(U)=\frac{\alpha^{4/3}}{8 b^{1/3}}\int |U|^{8/3} + \frac{1}{4}\int |U|^4>0,
\end{equation}
thus $\inf_{\mathcal{N}} A \geq 0$. 

\smallbreak

\noindent 5) We are now ready to prove the existence of a minimizer using direct methods. Take a minimizing sequence $(U_n)\subset H^1(\R)$: $U_n\in \mathcal{N}$ such that $A(U_n)\to \inf_{\mathcal{N}} A$. From \eqref{eq:A(U)_alternatize}, this sequence is bounded  in $L^{8/3}(\R)$ and $L^{4}(\R)$; since $U_n\in \mathcal{N}$, then the sequence is also bounded in $H^1(\R)$. Take the Schwarz symmetrization $|U_n|^*$, and let $t_n>0$ be such that $t_n |U_n|^*\in \mathcal{N}$ (recall Step 1). Since $\||U_n|^*\|_{H^1}\leq \|U_n\|_{H^1}$ and $\||U_n|^*\|_{p}=\|U_n\|_p$ for $p\geq 1$, then $A'(|U_n|^*)|U_n|^*\leq A'(U_n)[U_n]= 0$, and $t_n\leq 1$ by Step 1. From \eqref{eq:A(U)_alternatize}, we see directly that $J(t_n |U_n|^*)\leq J(U_n)$. So, $(t_n |U_n|^*)$ is also a minimizing sequence, being radially decreasing, nonnegative, and bounded in $H^1(\R)$. We denote this new sequence again by $U_n$. 

In conclusion, there exists $U\in H^1_{rd}(\R)$, nonnegative, such that (up to a subsequence) $U_n\rightharpoonup U$ weakly in $H^1(\R)$. From \eqref{compactemb}, the converge is strong in $L^p(\R)$, for every $p>2$. Step 3 yields that $U\not\equiv 0$.

Finally, since $A(U)[U]\leq 0$, we may take $0<t\leq 1$ such that $tU\in \mathcal{N}$, and 
\begin{align*}
\inf_\mathcal{N} A \leq A(tU)  &=\frac{\alpha^{4/3} t^{8/3}}{8 b^{1/3}}\int |U|^{8/3} + \frac{t^{4}}{4}\int |U|^4 \leq \frac{\alpha^{4/3}}{8 b^{1/3}}\int |U|^{8/3} + \frac{1}{4}\int |U|^4\\
		 &=\lim_n \frac{\alpha^{4/3}}{8 b^{1/3}}\int |U_n|^{8/3} + \frac{1}{4}\int |U_n|^4=\lim_n A(U_n)=\inf_\mathcal{N} A.
\end{align*}
In particular $t=1$, $U\in \mathcal{N}$ and $A(U)=\inf_\mathcal{N} A$. By Step 2 we deduce that $A'(U)=0$, that is, $U$ solves \eqref{eq:singleequation1}. Since $U\not\equiv 0$, then $U>0$ by the strong maximum principle.

\smallbreak

\noindent 6) In conclusion, the pair $(U,V)$, for $U$ positive solution of  \eqref{eq:singleequation1} and $V=-\left(\frac{\alpha}{b}\right)^{1/3} U^{2/3}$ solve \eqref{eq:boundstates4}, which proves Theorem \ref{teo:GS2}-1.
\end{proof}

\subsection*{Existence for  $\mathbf{a>0}$ and $\mathbf{b=0}$ (Theorem \ref{teo:GS2}-2.)}

In this case  the second equation in \eqref{eq:boundstates4} yields
\[
V(x)=\pm \left(\frac{\alpha}{a}\right)^{1/2} U(x),
\]
and for this reason we obtain two pairs of solutions.

The proof of Theorem \ref{teo:GS2}-2 follows the lines of the previous case $a=0$, $b>0$, with very few changes. We are lead this time to the problems
\begin{equation}\label{eq:equation_pm}
-U''+\alpha \rho U= \pm \left(\frac{\alpha}{a}\right)^{1/2}U^2+U^3,\qquad U\in H^1(\R),
\end{equation}
with associated action functionals
\[
A_{\pm}(U):=\frac{1}{2}\int \left( (U')^2+\alpha \rho U^2\right)\mp \int \frac{1}{3}\left(\frac{\alpha}{a}\right)^{1/2}U^3-\frac{1}{4}\int |U|^4.
\]
Unlike the sign of the cubic term, the sign of the quadratic term $\left(\frac{\alpha}{a}\right)^{1/2}s^2$ in \eqref{eq:equation_pm} is not important: since it is an $\text{o}(s)$ as $s\to 0$, and is dominated in absolute value by $C(1+|s|^3)$ for all $s\in \R$, the proof is analogous to the one of Theorem \ref{teo:GS2}-1. Solutions are critical points associated to the critical levels 
\[
c_{\pm}=\{U\in H^1_{\rm rad}(\R):\ U\neq 0,\  A'_\pm(U)[U]=0\}. \qedhere
\]

\subsection*{Existence for $\mathbf{a,b>0}$ and $\alpha>0$ small}

The remainder of the paper is  devoted to the proof of Theorem \ref{teo:GS2}-3.  This result follows directly from Theorem \ref{thm:auxteo1} and Theorem \ref{thm:auxteo2} below. First, we prove the existence of a solution whose components have different signs.

\begin{Teorema}\label{thm:auxteo1}  Take $\rho,a,b>0$. Then, for sufficiently small $\alpha>0$, \eqref{eq:boundstates4} admits a solution $(U,V)$, with $U\in H^1(\R)$ and $U>0$, $V<0$ in $\R$.
\end{Teorema}

\medbreak

In order to prove this result, let $\tilde f:= f|_{]-\infty,0[}$, where $f(t)=at^2-bt^3$, and  take the function
\begin{equation}\label{eq:definition_g}
g(s)=\begin{cases}
\tilde f^{-1}(s) & s>0\\
0	& s\leq 0.
\end{cases}
\end{equation}
which is negative for $s>0$. The asymptotic behavior at the origin and at plus infinity is
\begin{equation}\label{eq:lim_g(s)2}
\lim_{s\to 0^+} \frac{g(s)}{s^{1/2}}= -\frac{1}{\sqrt{a}},\qquad  \lim_{s\to +\infty} \frac{g(s)}{s^{1/3}}=  -\frac{1}{\sqrt[3]{b}}.
\end{equation}
In particular, there exists $c_1,c_2>0$, depending only on $a,b$, such that
\begin{equation}\label{eq:g(s)}
 |g(s)| \leq  c_1 (s^{1/2}+s^{1/3})
\end{equation}
\begin{equation}\label{eq:G(s)}
|G(s)| \leq  c_2 (s^{3/2}+s^{4/3})
\end{equation}
for every $s>0$. Consider the problem
\begin{equation}\label{eq:originalequation}
-U''+ \alpha 
\rho U=-g(\alpha U^2) \alpha U +U^3,\quad U\in H^1(\R),
\end{equation}
and the associated functional
\[
A(U):=\frac{1}{2}\int \left( (U')^2+\alpha \rho U^2\right) + \frac{1}{2}\int G(\alpha (U^+)^2)-\frac{1}{4}\int (U^+)^4, \quad U\in H^1(\R),
\]
where $G(t)=\int_0^t g(\xi)\, d\xi$. 

\begin{Lemma}\label{lemma:lastaux1}  Nontrivial critical points of $A$ are positive solutions of \eqref{eq:originalequation}. 
\end{Lemma}
\begin{proof} If $A'(U)=0$ with $U\not\equiv 0$, then 
\[
-U''+ \alpha \rho U=- g(\alpha (U^+)^2) \alpha U^++(U^+)^3.
\]
Multiplying this equation by $U^-$ and integrating by parts yields
\[
\int ((U')^2+\rho \alpha (U^-)^2)=0,
\] so $U^-\equiv 0$ and $U\geq 0$. Since $U\not\equiv 0$,  the strong maximum principle implies that $U>0$, hence $U$ is a positive solution of \eqref{eq:originalequation}. 
\end{proof}

\begin{Lemma}\label{lemma:PS_aux} There exists $(U_n)\subset H^1(\R)$ and $c>0$ such that
\begin{equation}\label{eq:Palais-Smale}
A(U_n)\to c,\qquad A'(U_n)\to 0 \quad \text{ in } H^{-1}(\R).
\end{equation}
\end{Lemma}
\begin{proof}
Recalling \eqref{eq:lim_g(s)2}, we get the existence of $c_3, c_4>0$, depending on $a$ and $b$, such that for every $s>0$
\begin{equation}\label{eq:g(s)_extra}
 g(s) \geq -c_3(s^{1/2}+s^{1/3}),\qquad   G(s)\geq -c_4(s^{3/2}+s^{4/3}).
\end{equation}

Let us check that the functional $A$ satisfies all the assumptions of the Mountain Pass Lemma (we will use the version from  \cite[Theorem 1.15]{MW}, which does not require that $A$ satisfies the Palais-Smale condition, and whose conclusion is precisely \eqref{eq:Palais-Smale}):
\begin{itemize}
\item $A(0)=0$
\item We have, denoting by $S_p$ the best Sobolev constant of the continuous embedding $H^1(\R)\hookrightarrow L^p(\R)$ and using \eqref{eq:g(s)_extra}:
\[
\begin{split}
A(U) \geq &\frac{\min \{1,\alpha\rho\}}{2} \|U\|_{H^1}^2 -\frac{c_3}{2}\int (\alpha^{3/2}|U|^3+\alpha^{4/3}|U|^{8/3})-\frac{1}{4}\int U^4\\
	\geq &\frac{\min \{1,\alpha\rho\}}{2} \|U\|_{H^1}^2 \\
			&-\frac{c_3}{2}\left( \alpha^{3/2}S_3^3 \|U\|_{H^1}^3 +\alpha^{4/3} S_{8/3}^{8/3} \|U\|^{8/3}_{H^1}\right)-\frac{S_4^4}{4}\|U\|^4_{H^1}\\
\end{split}
\]
and thus there exists $\varepsilon>0$ small (depending on $\alpha$), we have
\[
\inf_{\|U\|_{H^1}=\varepsilon}A(U)>0.
\]
\item Let $w\in H^1(\R)$ be a positive function. Then, by reasoning exactly as in point 1) of the proof of Lemma \ref{lemma:minimization_M3}, we deduce that
\[
A(tw)\to -\infty \qquad \text{ as } t\to +\infty.
\]
In conclusion, there exists $\bar U\in H^1 (\R)$ with $\|\bar U\|_{H^1}>\varepsilon$ such that $A(\bar U)<0$.
\end{itemize}
Thus, \cite[Theorem 1.15]{MW} applies, yielding the existence of a sequence $(U_n)\subset H^1(\R)$ satisfying \eqref{eq:Palais-Smale}.
\end{proof}

\begin{Lemma}\label{lemma:positivesolution} Equation \eqref{eq:originalequation} admits a positive solution $U\in H^1(\R)$.
\end{Lemma}
\begin{proof} Let $U_n$ be the sequence given by Lemma \ref{lemma:PS_aux}.
\smallbreak

\noindent 1) Let us check that $(U_n)$ is bounded in $H^1(\R)$. We have 
\[
\frac{1}{2}\int ((U'_n)^2+\alpha \rho U_n^2)+\frac{1}{2}\int G(\alpha (U_n^+)^2)-\frac{1}{4}\int (U_n^+)^4\leq C
\]
and
\[
\int ((U_n')^2+\alpha \rho U_n^2) +\int g(\alpha (U_n^+)^2) \alpha (U_n^+)^2 -\int (U_n^+)^4=\text{o}(\|U_n\|).
\]
By multiplying the second equation by $3/8$ and subtracting it from the first inequality, we have, using also \eqref{eq:g(s)} and \eqref{eq:G(s)}
\begin{align*}
\frac{1}{8}&\int ((U_n')^2+\alpha \rho U_n^2) \\
&\leq \, C+\text{o}(\|U_n\|)-\frac{1}{8}\int (U_n^+)^4 + \int \left(\frac{3}{8}g(\alpha (U_n^+)^2)\alpha (U_n^+)^2-\frac{1}{2}G(\alpha (U_n^+)^2)\right)\\
	& \leq \, C+\text{o}(\|U_n\|_{H^1(\R)})-\frac{1}{8}\int (U_n^+)^4  + C \int (\alpha^{3/2}+\alpha^{4/3})\left((U_n^+)^2+(U_n^+)^4\right)
\end{align*}
for some $C$ depending only on $c_3$ (and, thus, only on $a$ and $b$). Therefore, 	
\begin{align*}
\frac{\min\{1,\rho\alpha\}}{8}\| U_n\|^2_{H^1} \leq & C+C'\|U_n\|_{H^1} +  C'' \int (\alpha^{3/2}+\alpha^{4/3}) \|U_n\|_{H^1}^2 \\
								&  +\left(C (\alpha^{3/2}+\alpha^{4/3})-\frac{1}{8}\right) \int U_n^4.
\end{align*}
Choosing $\alpha$ sufficiently small such that $C''(\alpha^{3/2}+\alpha^{4/3}) \leq \frac{\min\{1,\rho\alpha\}}{16} $ and $C (\alpha^{3/2}+\alpha^{4/3})-\frac{1}{8}\leq 0$, we have
\[
\frac{\min\{1,\rho\alpha\}}{16}\int ((U_n')^2+\alpha \rho U_n^2) \leq C+C'\|U_n\|_{H^1},
\]
hence $(U_n)$ is a bounded sequence in $H^1(\R)$.
\smallbreak

\noindent 2) From the previous step, there exists $U\in H^1(\R)$ such that, up to a subsequence, 
\begin{align*}
& U_n\rightharpoonup U \text{ weakly in  } H^1(\R)
& U_n \to U \text{ strongly in  } L^p_{\rm loc}(\R),\ \forall p\geq 1.
\end{align*} Then, for every $\varphi\in C^\infty_{\rm c}(\R)$, we have $A'(U)[\varphi]=\lim_n A(U_n)[\varphi]=0$, and $U$ is a critical point of $A$, i.e., $A'(U)= 0$.

\smallbreak

 \noindent 3) To conclude, let us show that without loss of generality we can assume $U\not \equiv 0$. If this is true, Lemma \ref{lemma:positivesolution} follows directly from Lemma \ref{lemma:lastaux1}.  
 
 The first observation is that, since $U_n$ is bounded in $H^1(\R)$ and $A'(U_n)[U_n]=\textrm{o}(\|U_n\|_{H^1})$, then actually 
 \begin{equation}\label{eq:auxiliary_}
 A'(U_n)[U_n]=\textrm{o}(1).
 \end{equation}  
This implies that we cannot have $U_n \to 0$ in $L^4(\R)$, otherwise \eqref{eq:auxiliary_} combined with \eqref{eq:g(s)} would yield $U_n\to 0$ in $H^1(\R)$ and $A(U_n)\to 0$, contradicting the positivity of $c>0$ in Lemma \ref{lemma:PS_aux}.
Therefore $U_n\not\to U$ in $L^4(\R)$, and since $U_n$ is bounded in $H^1(\R)$ there exists $R$, $x_n\in \R$ and $l>0$ such that
 \[
 \int_{B_R(x_n)} U_n^4 \geq l>0
 \]
 (check for instance Lemma 1.21 in \cite{MW}). Thus, defining $V_n(x)=U_n(x-x_n)$, we have
 \[
 A(V_n)=A(U_n)\to c,\qquad A'(V_n)\to 0,\qquad \int_{B_R(0)} V_n^4 \geq l>0.
 \]
By repeating the previous arguments we obtain the existence of $V\in H^1(\R)$ such that $V_n\to V$ weakly in $H^1(\R)$, strongly in $L^p_{loc}(\R)$ (up to a subsequence).  Moreover, $A'(V)=0$ and $\int_{B_R(0)}V^4\geq l>0$, hence $V$ is a nontrivial critical point of $A$, hence a positive solution of \eqref{eq:originalequation} by Lemma \ref{lemma:lastaux1}.
\end{proof}

\medbreak

\begin{proof}[Conclusion of the proof of Theorem \ref{thm:auxteo1}]
Let $U\in H^1(\R)$ be the positive solution of \eqref{eq:originalequation} provided by the previous lemma. Let $V:=g(\alpha U^2)$. Then, by definition of $g$ (recall \eqref{eq:definition_g}), we have $V<0$ and $V^2-bV^3=\alpha U^2$. In particular, $(U,V)$ is a solution of \eqref{eq:boundstates4}.
\end{proof}

Having concluded the proof of Theorem \ref{thm:auxteo1}, we turn to the  existence of a solution pair with both components positive.

\begin{Teorema}\label{thm:auxteo2}  Take $\rho,a,b>0$. Then, for sufficiently small $\alpha>0$, \eqref{eq:boundstates4} admits a solution pair $(U,V)$, with $U\in H^1(\R)$ and $U>0$, $V>0$ in $\R$.
\end{Teorema}

In order to prove this last result, consider this time $\tilde f:=f|_{[0,\frac{2a}{3b}]}$, where we recall that $f(s)=as^2-bs^3$. This function is strictly increasing in $[0,\frac{2a}{3b}]$, hence invertible. Since $f(\frac{2a}{3b})=\frac{4a^3}{27 b^2}$,  we can take the continuous, nonnegative functions
\begin{equation}
\tilde h(s)=\begin{cases}\label{eq:auxiliaryH}
0 &  s\leq 0\\
\tilde f^{-1}(s)& 0 \leq s\leq \frac{4a^3}{27 b^2}\\
\frac{2a}{3b} & s\geq \frac{4a^3}{27 b^2}
\end{cases}; \qquad h(s)=\min \{\tilde h(s),\frac{\rho}{2}\}. 
\end{equation}
We will solve
\begin{equation}\label{eq:lastequation_extra}
-U'' + \alpha\rho U +h(\alpha U^2)\alpha U=U^3,
\end{equation}
obtaining positive solutions as critical points of the functional $A:H^1(\R)\to \R$ which now is defined by
\begin{equation}\label{def_of_A}
A(U)= \frac{1}{2}\int \left((U')^2+\rho \alpha U^2\right) +\frac{1}{2}\int H(\alpha U^2)-\frac{1}{4}\int (U^+)^4,
\end{equation}
where $H(t):=\int_0^t h(\xi)\, d\xi$. Since $0\leq h(t)\leq \min\{\frac{2a}{3b},\frac{\rho}{2}\}$, then $0\leq H(\alpha t^2)\leq \min\{\frac{2a}{3b},\frac{\rho}{2}\} \alpha t^2$ and the functional is well defined.  Observe also that, if $U>0$ is such that $\alpha U^2 \leq \frac{4a^3}{27 b^2}$, then $V:=h (\alpha U^2)>0$ and $aV^2-bV^3=\alpha U^2$.

\begin{Lemma}\label{lemma:lastaux2}  Nontrivial critical points of $A$ are positive solutions of \eqref{eq:originalequation}. 
\end{Lemma}
\begin{proof} If $A'(U)=0$ with $U\not\equiv 0$, then 
\[
-U''+ \alpha \rho U +h(\alpha U^2)\alpha U= (U^+)^3.
\]
Multiplying this equation by $U^-$ and integrating by parts, we obtain 
\[
\int ((U')^2+\rho \alpha (U^-)^2) + \int h(\alpha U^2)\alpha U^2=0.
\]
Since $h\geq 0$,  we have $U^-\equiv 0$ and $U\geq 0$, and the conclusion follows from the strong maximum principle. \end{proof}

\begin{Lemma}\label{lemma_criticalpointU}
There exists $U> 0$, a critical point of $A$, such that
\[
A(U)\leq c:= \inf_{\gamma\in \Gamma} \sup_{t\in [0,1]} A(\gamma(t)),
\]
where
\[
\Gamma:=\{\gamma\in C([0,1],H^1(\R)):\ \gamma(0)=0,\ A(\gamma(1))<0\}.
\]
\end{Lemma}
\begin{proof}
Similarly to the proof of Lemma \ref{lemma:PS_aux}, let us check that the functional \eqref{def_of_A} satisfies the assumptions of \cite[Theorem 1.15]{MW}:
\begin{itemize}
\item $A(0)=0$
\item Since $H\geq 0$, we have
\[
A(U) \geq \frac{\min \{1,\alpha\rho\}}{2} \|U\|_{H^1}^2 -C\|U\|^4_{H^1}
\]
and thus, for $\varepsilon>0$ small, 
\[
\inf_{\|U\|_{H^1}=\varepsilon}A(U)>0.
\]
In particular, since $c\geq \inf_\mathcal{N} A$, this implies that $c> 0$.
\item Taking a positive function $w\in H^1(\R)$ we have, since $H(\alpha t^2)\leq \frac{\rho}{2} \alpha t^2$,
\[
A(tw)\leq \frac{t^2}{2}\int ((w')^2 + \frac{3\rho}{2}\alpha w^2) - \frac{t^4}{4}\int w^4\to -\infty
\]
as $t\to +\infty$.
\end{itemize}
Thus, \cite[Theorem 1.15]{MW} implies the existence of a sequence $(U_n)\subseteq H^1(\R)$ such that
\[
A(U_n)\to c>0,\qquad A'(U_n)\to 0.
\]
This sequence is bounded in $H^1(\R)$, since
\begin{align*}
C + \textrm{o}(\|U_n\|_{H^1}) &\geq A(U_n)-\frac{1}{4}A'(U_n)[U_n] \\
					&= \frac{1}{4}\int ((U_n')^2+\rho\alpha U_n^2) + \frac{1}{2}\int H(\alpha U_n^2) -\frac{1}{4}\int h(\alpha U_n^2)\alpha U_n^2 \\
					&\geq  \frac{1}{4}\int \Big((U_n')^2 +\alpha \underbrace{(\rho-h(\alpha U_n^2))}_{\geq \rho/2} U_n^2\Big). 
\end{align*}
Thus $U_n\rightharpoonup U$ weakly in $H^1(\R)$ (up to a subsequence), and since $A(U_n)\to c>0$, then $U_n \not \to 0$ in $L^4(\R)$. Reasoning exactly as in the proof of Lemma \ref{lemma:positivesolution}, we can assume without loss of generality that $U\not\equiv 0$, and $A'(U)=0$. Therefore $U$ is positive as a consequence of Lemma \ref{lemma:lastaux2}.

Finally,  since
\[
\begin{split}
A(U_n) =&\, A(U_n)-\frac{3}{8}A'(U_n)[U_n]+\textrm{o}(1)\\
		=&\, \frac{1}{8}\int \left((U_n')^2+(\rho-\frac{3}{4} h(\alpha U_n^2))\alpha U_n^2 \right) +\frac{1}{2}\int H(\alpha U_n^2)\\
			&+\frac{1}{8}\int (U_n^+)^4 + \textrm{o}(1)
\end{split}
\]
and all integrands are nonnegative for $\alpha$ sufficiently small, from Fatou's lemma  conclude that
\[
\begin{split}
c &=\lim_n A(U_n) \geq \liminf \left( A(U_n)-\frac{3}{8}A'(U_n)[U_n]+\textrm{o}(1)\right) \\
	&\geq \liminf ( A(U)-\frac{3}{8}A'(U)[U])=A(U). \qedhere
\end{split}
\]

\end{proof}

Up to this point, we have obtained a positive solution of the equation \eqref{eq:lastequation_extra}. In order to conclude the proof of Theorem \ref{thm:auxteo2}, we need to show that $\alpha U^2\leq \frac{4a^3}{27 b^2}$ for $\alpha$ small (observe that $U$ depends on $\alpha$, so this is a delicate step).
Havind that in mind, consider the auxiliary functional
\[
\tilde A(U)= \frac{1}{2}\int \left( (U')^2 + \frac{3\rho}{2}\alpha U^2\right)-\frac{1}{4}\int (U^+)^4,
\]
which satisfies:
\begin{equation}\label{eq:monotonicity}
A(U)\leq \tilde A(U),\qquad \forall U\in H^1(\R).
\end{equation}
It is classical to see (see for e.g. \cite{MW}) that $\tilde A$  admits the following (least action) critical level in $H^1(\R)$:
\[
c_{\tilde A}=\inf_{u\in H^1(\R)} \sup_{t>0} \tilde{A}(tu)= \inf_{\mathcal{N}_{\tilde A}} \tilde{A},
\]
where
\[
\mathcal{N}_{\tilde A}=\{u\in H^1(\R):\ u\neq 0,\ \tilde A'(u)u=0\},
\]
which is  achieved by a unique (up to translation) radial positive solution of 
\[
-W''+\alpha \frac{3\rho}{2}W=W^3 \text{ in } \R,\qquad W(x)\to 0 \text{ as } |x|\to \infty.
\]
This solution is explicitly known to be
\[
W(x)=\sqrt{2}\sqrt{\alpha \frac{3\rho}{2}}\, \sech(\sqrt{\alpha \frac{3\rho}{2}}\, x),
\]
so that there exists $\kappa>0$ independent of $\rho$ and $\alpha$ such that
\begin{equation}\label{estimate}
0<c_{\tilde A}=A(W)= \frac{1}{4}\int ((W')^2 +\frac{3\rho}{2}\alpha W^2) =\frac{1}{4}\int W^4 \leq \kappa \left(\frac{3\rho}{2}\alpha\right)^{3/2}
\end{equation}

\begin{Lemma}\label{4.9}
Let $U$ be the critical point obtained in Lemma \ref{lemma_criticalpointU}. Then we have $A(U)\leq c_{\tilde A}$.
\end{Lemma}

\begin{proof}
We have, by  \eqref{eq:monotonicity},
\[
A(U)\leq c\leq  \sup_{t>0} A(tW)\leq \sup_{t>0} \tilde A(tW)=\tilde A(W)=c_{\tilde A},
\]
and the conclusion follows.
\end{proof}

\medbreak

\begin{proof}[Conclusion of the proof of Theorem \ref{thm:auxteo2}]
Let $U$ be the critical point of $A$ obtained in Lemma \ref{lemma_criticalpointU}. Then
\[
A(U)=A(U)-\frac{1}{4}A'(U)[U] \geq  \frac{1}{4}\int ((U')^2+ \frac{\rho}{2} \alpha U^2)
\]
and, by combining \eqref{estimate} with Lemma \ref{4.9} and since the embedding $H^1(\R)\hookrightarrow L^\infty(\R)$ is continuous,
\[
\|U\|^2_\infty \leq C_1 \|U\|_{H^1}^2  \leq  \frac{C_1}{\sqrt{\rho \alpha}} \int \left((U')^2+\rho \alpha U^2\right) \leq \frac{C_2}{\sqrt{\rho \alpha}}A(U) \leq C_3 \rho \alpha,
\]
for $C_1,C_2,C_3$ independent of $a,b,\rho,\alpha$. Now choose $\alpha$ small so that
\[
\alpha \|U\|_\infty^2 \leq  C_3^2\rho \alpha \leq \frac{4a^3}{27 b^2} 
\]
and $\tilde h(\alpha U^2)\leq \frac{\rho}{2}$ (recall \eqref{eq:auxiliaryH}). Then $V:=h(\alpha U^2)={\tilde f}^{-1}(\alpha U^2)$ satisfies $aV^2-bV^3=\alpha U^2$, and $(U,V)$ is a solution of 
\eqref{eq:boundstates4}.
\end{proof}

\medbreak

\begin{proof}[Conclusion of the proof of Theorem \ref{teo:GS2}-3.] This result is a direct consequence of Theorem \ref{thm:auxteo1} and Theorem \ref{thm:auxteo1}.
\end{proof}
\noindent {\bf Acknowledgements.}  The authors are grateful to Hermano Frid and Vla\-di\-mir Ko\-no\-top for stimulating suggestions and to the anonymous referees for their insightful observations that helped to considerably improve the present paper. \newline
 J. P. Dias was supported by FCT (Portugal) project UID/MAT/\-04561/\-2013. \newline Filipe Oliveira was partially supported by the Project CEMAPRE-UID/-
 \noindent
 MULTI/00491/2013 financed by FCT/MCTES through national funds. \newline H. Tavares was partially supported by FCT (Portugal) through the projects UID/MAT/04459/2013 and UID/MAT/04561/2013 and the program \textit{Investigador FCT}, as well as by the ERC Advanced Grant 2013 n. 339958 ``Complex Patterns for Strongly Interacting Dynamical Systems - COMPAT''.


\begin{thebibliography}{xxxx89}

\bibitem{AK} I.\-S. Aranson, L. Kramer; {\it The world of the complex Ginzburg-Landau equation}, Rev. Modern Phys. 74 (2002), no. 1, 99--143. 

\bibitem{B} {D.\-J. Benney}; {\it A general theory for interactions between short and long waves}, Studies in Appl. Math., 56 (1977), 81--94.

\bibitem{CDW2} T. Cazenave, F. Dickstein, F.\-B. Weissler; {\it Finite-time blowup for a complex Ginzburg-Landau equation}, SIAM J. Math. Anal. 45 (2013), no. 1, 244--266-

\bibitem{CDW}  T. Cazenave, F. Dickstein, F.\-B. Weissler; {\it Standing waves of the complex Ginzburg-Landau equation}, Nonlinear Anal. 103 (2014), 26--32.

\bibitem{CDP} R. Cipolatti, F. Dickstein, J.-P. Puel; {\it Existence of standing waves for the complex Ginzburg-Landau equation}, J. Math. Anal. Appl. 422 (2015), no. 1, 579--593.


\bibitem{DF} J.-P. Dias, M. Figueira; {\it Existence of weak solutions for a quasilinear version
of Benney equations},  J. Hyperbolic Diff. Eq.  4 (2007), 555--563.

\bibitem{DFF} J.-P. Dias, M. Figueira, H. Frid; \textit{Vanishing viscosity with short wave long wave interactions for systems of conservation laws,}   Arch. Rational Mech. Anal. 196 (2010) 981--1010.

\bibitem{DFO} {J.-P. Dias, M. Figueira,  F. Oliveira}; \textit{Existence of local strong solutions for a quasilinear Benney system},  C. R. Math. Acad. Sci. Paris 344 (2007), no. 8, 493--496.

\bibitem{DFO1} {J.-P. Dias, M. Figueira,  F. Oliveira}; \textit{Existence and linearized stability of solitary waves for a quasilinear Benney system}, Proc. Royal Soc. Edinburgh, Section A, 146 (2016), 547-564.

\bibitem{GR} E. Godlewski, P.\-A. Raviart;  Hyperbolic systems of conservation laws, Ellipses, Paris, 1991. 

  \bibitem{K} T. Kato, Quasi-linear equations of evolution, with applications to partial differential equations, Lecture Notes in Mathematics, Springer-Verlag, New-York, 448 (1975), 25-70. 

 \bibitem{Lions1} P.~L. Lions, The concentration-compactness principle in the calculus of variations, Part 1, Ann. Inst. H. Poincar\'e {\bf 1} (1984), 109--145.
 
\bibitem{M} {F. Murat}; {\it Compacit\'e par compensation,} Ann. Scuola Norm. Sup. Pisa 5 (1978), 489--507.

\bibitem{MS} {M.E. Schonbek}; {\it Convergence of solutions to nonlinear dispersive equations},  Comm. Partial Diff. Eq. 7 (1982), 959--1000.

 \bibitem{O} F. Oliveira; {\it Stability of the solitons for the one-dimensional Zakharov-Rubenchik equation}, Phys. D 175 (2003), no. 3-4, 220--240.

\bibitem{S} {S. Snoussi}; {\it On the local and global existence of solution for a general Ginzburg-Landau like equation coupled with a Poisson equation in $L^p(\R^d)$}, Differential Integral Equations 13 (2000), no. 1-3, 61--98. 

\bibitem{ST} Y. Shibata, Y. Tsutsumi; {\it Local existence of solution for the initial-boundary value problem of fully nonlinear wave equation}, Nonlinear Anal. 11 (1987), no. 3, 335--365.

\bibitem{Strauss} W.~A. Strauss, Existence of solitary waves in higher dimensions, Comm. Math. Phys. {\bf 55} (1977), 149--162.

\bibitem{T} {L. Tartar}; {\it Compensated compactness and applications to partial differential equations,} Research Notes in Math. 39, pg.136--212, Pitman, 1979.

\bibitem{TH1} {M. Tsutsumi, S. Hatano;} {\it Well-posedness of the Cauchy problem for the long wave - short wave resonance equations,} Nonlinear Anal. 22 (1994), 155-171.

\bibitem{TH2} M. Tsutsumi, S. Hatano;  {\it Well-posedness of the Cauchy problem for Benney's first equations of  long wave short wave interactions,}  Funkcialaj Ekvacioj 37 (1994), 289--316.

\bibitem{TK} M. Tsutsumi, H. Kasai; {\it The time-dependent Ginzburg-Landau Maxwell equations}, Nonlinear Anal. 37 (1999), no. 2, 187--216.

\bibitem{a21} M. Funakoshi and M. Oikawa, The resonant interaction betwee
n
a long internal gravity wave and a surface gravity wave packet,
J. Phys. Soc. Japan 52 (1983), 1982-1995.
\bibitem{a22}  R. H. J. Grimshaw, The modulation of an internal gravity-wave
packet and the resonance with the mean motion, Stud. Appl.
Math. 56 (1977), 241-266.
\bibitem{a36} V.D. Djordjevic, L.G. Redekopp, On the two-dimensional pack-
ets of capillary-gravity waves, J. Fluid. Mech. 79 (1977) 703-714

 \bibitem{MW} M. Willem; Minimax Theorems, Birkh\"auser 1996.
 
 \bibitem{Aranson} Aranson, Kramer, Rev. Med. Phys. 74 (2002)). 
\end{thebibliography}
\end{document}